\documentclass[12]{amsart}
\usepackage{amsmath,amssymb,amsthm,color,enumerate,comment,centernot,enumitem,url,cite}
\usepackage{graphicx,relsize,bm}
\usepackage{mathtools}
\usepackage{array}

\makeatletter
\newcommand{\tpmod}[1]{{\@displayfalse\pmod{#1}}}
\makeatother

\newtheorem{thm}{Theorem}[section]
\newtheorem{lemma}[thm]{Lemma}

\newtheorem{cor}[thm]{Corollary}

\theoremstyle{remark}

\theoremstyle{definition}

\theoremstyle{THM}

\newcommand{\G}{{\mathcal G}}

\newcommand{\FF}{{\mathcal F}}

\newcommand{\rad}{{\mbox{{\rm{rad}}}}}
\newcommand{\Mod}[1]{\ (\mathrm{mod}\enspace #1)}
\newcommand{\mmod}[1]{\ \mathrm{mod}\enspace #1}
\newcommand{\Z}{{\mathbb Z}}
\newcommand{\Q}{{\mathbb Q}}

\newcommand{\F}{{\mathbb F}}
\newcommand{\abs}[1]{\left|{#1}\right|}

\newcommand{\ds}{\displaystyle}
\newcommand{\D}{{\mathcal D}}

\makeatletter
\@namedef{subjclassname@2020}{%
  \textup{2020} Mathematics Subject Classification}
\makeatother

\def\red#1 {\textcolor{red}{#1 }}
\def\blue#1 {\textcolor{blue}{#1 }}

\numberwithin{equation}{section}

\begin{document}

\title[Generalized Wieferich primes and monogenic polynomials]{Generalized Wieferich primes\\ and monogenic polynomials}

\author{Lenny Jones}
\address{Professor Emeritus, Department of Mathematics, Shippensburg University, Shippensburg, Pennsylvania 17257, USA}
\email[Lenny~Jones]{doctorlennyjones@gmail.com\\ ORCID: 0000-0001-7661-4226}

\date{\today}

\begin{abstract}
Let $b, p\in \Z$ with $b\ge 2$ and $p\ge 3$ a prime. If $b^{p-1}\equiv 1 \pmod{p^2}$, then $p$ is called a {\em generalized Wieferich prime base $b$}, or more succinctly, a {\em base-$b$ Wieferich prime}. When $b=2$, $p$ is also known simply as a Wieferich prime.  
We say that a monic polynomial $f(x)\in \Z[x]$ is {\em monogenic} if $f(x)$ is irreducible over ${\mathbb Q}$ and $\{1,\theta,\theta^2,\ldots,\theta^{\deg(f)-1}\}$ is a basis for the ring of integers of ${\mathbb Q}(\theta)$, where $f(\theta)=0$. 

Recently, necessary and sufficient conditions for the monogenicity of the trinomials $x^{2n}+bx^n+b$ were given that included   base-$b$ Wieferich prime congruences, and also congruences involving a certain Lucas sequence. In this article, we prove a similar result for a different class of trinomials that does not rely on any congruence condition involving a Lucas sequence. Furthermore, we extend this result to a related class of $N$-nomials, for any $N\ge 4$.
\end{abstract}


\maketitle
\section{Introduction}\label{Section:Intro}
For an integer $b\ge 2$ and a prime $p\ge 2$, we say that $p$ is a {\em generalized Wieferich prime base $b$}, or more succinctly, a {\em base-$b$ Wieferich prime} if $b^{p-1}\equiv 1 \pmod{p^2}$  \cite{Conrad}. When $b=2$, the prime $p$ is also known simply as a {\em Wieferich prime}. For any particular $b$, it is still not known as to whether there exist infinitely many, only finitely many, or no  base-$b$ Wieferich primes \cite{Conrad,OEIS}. Observe that the case $p=2$ is somewhat trivial since it follows immediately from the definition that $2$ is a base-$b$ Wieferich prime if and only if $b\equiv 1 \pmod{4}$. 

We say that a monic polynomial $f(x)\in \Z[x]$ is {\em monogenic} if $f(x)$ is irreducible over ${\mathbb Q}$ and $\{1,\theta,\theta^2,\ldots,\theta^{\deg(f)-1}\}$ is a basis for the ring of integers of ${\mathbb Q}(\theta)$, where $f(\theta)=0$. It is well known \cite{Cohen} that 
  \begin{equation} \label{Eq:Dis-Dis}
\Delta(f)=\left[\Z_K:\Z[\theta]\right]^2\Delta(K),
\end{equation} where $\Delta(f)$ and $\Delta(K)$ denote the discriminants over $\Q$, respectively, of $f(x)$ and the number field $K$. We refer to $\left[\Z_K:\Z[\theta]\right]$ as the {\em index} of $f(x)$, and we denote it index($f$). 
Thus, from \eqref{Eq:Dis-Dis}, $f(x)$ is monogenic if and only if the index($f$)=1.

With $p$ an odd prime, the author has recently shown \cite{JonesWie} that the trinomial $x^{2p}+2x^p+2$ is monogenic if and only if $p$ is not a Wieferich prime. The following theorem, which generalizes \cite{JonesWie}, is a direct consequence of the more recent work presented in \cite{FHJ}, and provides the primary motivation for the results in this article. 
\begin{thm}\label{Thm:FHJ}
Let $k\ge 3$ be an integer, and define
\[\G_{k,b}(x):=x^{2k}+bx^k+b.\] Let $b\ge 2$ be an integer such that $b$ and $b-4$ are squarefree, and $b\not \equiv 1 \pmod{4}$ if $k\equiv 0 \pmod{2}$.  Then $\G_{k,b}(x)$ is monogenic if and only if
\begin{equation*}\label{condition delta=1}
     (2V_p+b^p+b)^2-D(b^{p-1}-1)^2 \not \equiv 0 \pmod{p^3},
     \end{equation*} for every odd prime divisor $p$ of $k$ with $\delta=1$, and 
     \begin{equation*}\label{condition delta=-1} 
    V_p\not \equiv -b \pmod{p^2} \quad \mbox{ or } \quad b^{p-1}\not \equiv 1 \pmod{p^2}, 
    \end{equation*} for every odd prime divisor $p$ of $k$ with $\delta=-1$,
    where $D=b^2-4b$, $\delta$ is the Legendre symbol $\left(\frac{D}{p}\right)$ and $(V_i)_{i\ge -1}$ is the sequence defined by  
  \begin{gather*}
 V_{-1}:=-1, \quad V_0=2,\quad V_1=-b \quad \mbox{and}\\
  V_{i}=-bV_{i-1}-bV_{i-2} \quad \mbox{for $i\ge 2$.}
\end{gather*}  
   \end{thm} 
   We point out that both \cite{JonesWie} and Theorem \ref{Thm:FHJ} deal specifically with  
   trinomials. With  Theorem \ref{Thm:FHJ} as a motivation, there are two main goals of this article. The first goal is to determine a class of trinomials different from the trinomials in Theorem \ref{Thm:FHJ}, such that the conditions for monogenicity do not involve checking some congruence conditions in a Lucas sequence, while maintaining the generalized Wieferich prime conditions. The second goal is to generalize, in a straightforward way, this new class of trinomials to $N$-nomials, where $N\ge 4$. To accomplish these goals, we introduce the following framework. 

Throughout this article, for $n,m,b,k\in \Z$ with $n\ge 2$, $\abs{m}\ge 2$, $b\ge 2$ and $k\ge 1$, we assume that 
\begin{equation}\label{Defs1}
\begin{gathered}
 \gcd(b,m)=1, \quad b \ \mbox{is squarefree} \quad \mbox{and} \quad m\equiv 0 \Mod{\rad(k)},  
 \end{gathered}
\end{equation} 
where $\rad(1)=1$ and $\rad(k)$ is the product of the distinct prime divisors of the integer $k$ when $k\ge 2$. Furthermore, we suppose that 
\begin{equation}\label{Defs2}
  \begin{gathered}
   f_{n,b,m}(x):=x^n+b\sum_{j=1}^{n}(m^2x)^{n-j}, \quad 
  \FF_{n,b,m,k}(x):=f_{n,b,m}(x^k), \quad \mbox{and}\\ 
  d:=\frac{\abs{\D}}{\ds \prod_{\substack{p \ {\rm prime} \\ p\mid bm}} p^{\nu_p(\D)}} \quad \mbox{is squarefree},\\
    \mbox{where} \quad \D:=\frac{n^n(1-bm^{2n})^{n+1}+bm^{2n}(n+1)^{n+1}}{(1+bm^{2n}n)^2}\in \Z\\
    \mbox{and} \ \nu_p(z) \ \mbox{is the $p$-adic valuation of the integer $z$}. 
        \end{gathered}
    \end{equation}
To realize our goals using the above framework, we investigate exactly when the monogenicity of $\FF_{n,b,m,k}(x)$ is determined completely by whether certain prime divisors of $\Delta(\FF_{n,b,m,k})$ are, or are not, base-$b$ Wieferich primes. This analysis is divided into three parts: the case $n=2$, the case $n\ge 3$ with $k=1$, and the case $n\ge 3$ with $k\ge 2$ (respectively, Theorem \ref{Thm:Main1}, Theorem \ref{Thm:Main2} and Theorem \ref{Thm:Main3}).  
The case $n=2$ yields only trinomials, and so this case is more analogous to Theorem \ref{Thm:FHJ} in that regard. The approach used for $n\ge 3$ differs from the case $n=2$, and requires slightly more machinery. More precisely, the case $n\ge 3$ is divided into the two subcases $k=1$ and $k\ge 2$, where the subcase $k=1$  is addressed using a theorem of Jakhar, Kalwaniya and Yadev \cite{JKY1}, while the subcase $k\ge 2$ is handled using a power-compositional theorem due to Kaur, Kumar and Remete \cite{KKR}. 
 Our main results are as follows.
\begin{thm}\label{Thm:Main1}
 Assuming \eqref{Defs1} and \eqref{Defs2}, let $n=2$. 
Then the trinomial
\[\FF_{2,b,m,k}(x)=x^{2k}+bm^2x^k+b\] 
is monogenic if and only if no odd prime divisor of $k$ is a base-$b$ Wieferich prime and if, additionally, $m\equiv 0 \pmod{2}$, then 
$b\equiv 1 \pmod{4}$, or equivalently, 2 is a base-$b$ Wieferich prime.     
\end{thm}

The next theorem extends Theorem \ref{Thm:Main1} to $(n+1)$-nomials when $n\ge 3$ with $k=1$. 
\begin{thm}\label{Thm:Main2}
 Assuming \eqref{Defs1} and \eqref{Defs2}, let $n\ge 3$. Then $\FF_{n,b,m,1}(x)$ is monogenic if and only if no odd prime divisor of
  $\gcd(n,m)$ is a base-$b$ Wieferich prime and if, additionally, $\gcd(n,m)\equiv 0 \pmod{2}$, then 
$b\equiv 1 \pmod{4}$, or equivalently, 2 is a base-$b$ Wieferich prime.  
\end{thm}
The following corollary is an immediate consequence of Theorem \ref{Thm:Main2}.
\begin{cor}\label{Cor:Main2}
 Assuming \eqref{Defs1} and \eqref{Defs2}, let $n\ge 3$. If $\gcd(n,m)=1$, then $\FF_{n,b,m,1}(x)$ is monogenic.  
\end{cor} 

The next theorem extends Theorem \ref{Thm:Main2} to a power-compositional situation. 
\begin{thm}\label{Thm:Main3}
 Assuming \eqref{Defs1} and \eqref{Defs2}, let $n\ge 3$ and $k\ge 2$. 
 Assume further that 
 \[g_p(x):=x^{nk/p^{\nu_p(k)}}+b,\] is irreducible in $\F_p[x]$ for every prime divisor $p$ of $k$.   
  Then $\FF_{n,b,m,k}(x)$ is monogenic if and only if no odd prime divisor of $k\gcd(n,m)$  
  is a base-$b$ Wieferich prime and if, additionally,   
 $k\gcd(n,m)\equiv 0 \pmod{2}$, then $b\equiv 1 \pmod{4}$, or equivalently, 2 is a base-$b$ Wieferich prime.
\end{thm}

 \section{Preliminaries}\label{Section:Prelims}
 The first theorem is a standard tool used to determine if an irreducible polynomial is monogenic, and will be used in the proof of Theorem \ref{Thm:Main3}.
\begin{thm}[Dedekind's Index Criterion \cite{Cohen}]\label{Thm:Dedekind}
Let $K=\Q(\theta)$ be a number field, $T(x)\in \Z[x]$ the monic minimal polynomial of $\theta$, and $\Z_K$ the ring of integers of $K$. Let $p$ be a prime number and let $\overline{ * }$ denote reduction of $*$ modulo $p$ (in $\Z$, $\Z[x]$ or $\Z[\theta]$). Let
\[\overline{T}(x)=\prod_{i=1}^s\overline{\gamma_i}(x)^{e_i}\]
be the factorization of $T(x)$ modulo $p$ in $\F_p[x]$, and set
\[g(x)=\prod_{i=1}^s\gamma_i(x),\]
where the $\gamma_i(x)\in \Z[x]$ are arbitrary monic lifts of the $\overline{\gamma_i}(x)$. Let $h(x)\in \Z[x]$ be a monic lift of $\overline{T}(x)/\overline{g}(x)$ and set
\[F(x)=\dfrac{g(x)h(x)-T(x)}{p}\in \Z[x].\]
Then 
\[{\rm index}(T)\not \equiv 0 \pmod{p} \Longleftrightarrow \gcd\Bigl(\overline{F},\overline{g},\overline{h}\Bigr)=1 \mbox{ in } \F_p[x].\]
\end{thm}
The following theorem, 
which is an algorithmic version of Theorem \ref{Thm:Dedekind} formulated specifically for trinomials,  
will be used in the proof of Theorem \ref{Thm:Main1}. 
 \begin{thm}{\rm \cite{JKS2}}\label{Thm:JKS2}
Let $N\ge 2$ be an integer.
Let $K=\Q(\theta)$ be an algebraic number field with $\theta\in \Z_K$, the ring of integers of $K$, having minimal polynomial $f(x)=x^{N}+Ax^M+B$ over $\Q$, where $N_1=N/\gcd(M,N)$ and $M_1=M/\gcd(M,N)$. A prime factor $p$ of $\Delta(f)$ does not divide the {\rm index}($f$) 
if and only if $p$ satisfies one of the following conditions:
\begin{enumerate}[label=(\roman*), font=\normalfont]
  \item \label{JKS:C1} when $p\mid A$ and $p\mid B$, then $p^2\nmid B$;
  \item \label{JKS:C2} when $p\mid A$ and $p\nmid B$, then
  \[\mbox{either } \quad p\mid a_2 \mbox{ and } p\nmid b_1 \quad \mbox{ or } \quad p\nmid a_2\left((-B)^{M_1}a_2^{N_1}+\left(-b_1\right)^{N_1}\right),\]
  where $a_2=A/p$ and $b_1=\frac{B+(-B)^{p^j}}{p}$, such that $p^j\mid\mid N$ with $j\ge 1$;
  \item \label{JKS:C3} when $p\nmid A$ and $p\mid B$, then
  \[\qquad \quad \mbox{either}\quad p\mid a_1 \mbox{ and } p\nmid b_2  \quad\mbox{or}\quad  p\nmid a_1b_2^{M-1}\left((-A)^{M_1}a_1^{N_1-M_1}-\left(-b_2\right)^{N_1-M_1}\right),\]
  where $a_1=\frac{A+(-A)^{p^l}}{p}$, such that $p^l\mid\mid (N-M)$ with $l\ge 0$, and $b_2=B/p$;
  \item \label{JKS:C4} when $p\nmid AB$ and $p\mid M$ with $N=s^{\prime}p^k$, $M=sp^k$, $p\nmid \gcd\left(s^{\prime},s\right)$, then 
   \begin{equation*}
     x^{s^{\prime}}+Ax^s+B \quad \mbox{and}\quad \dfrac{Ax^{sp^k}+B+\left(-Ax^s-B\right)^{p^k}}{p} 
   \end{equation*} are coprime modulo $p$;
            \item \label{JKS:C5} when $p\nmid ABM$, then 
            \[p^2\nmid \left(B^{N_1-M_1}N_1^{N_1}-(-1)^{M_1}A^{N_1}M_1^{M_1}(M_1-N_1)^{N_1-M_1}\right).\]
   \end{enumerate}
\end{thm}

We use the following theorem for the proof of Theorem \ref{Thm:Main2}. 
\begin{thm}{\rm \cite{JKY1}}\label{Thm:JKY}
Let $K=\Q(\theta)$ be an algebraic number field with $\theta$ in the ring $\Z_K$ of algebraic integers of $K$ having minimal polynomial $f(x)=x^n+b\sum_{i=1}^n(ax)^{n-i}$ over $\Q$ with $n\ge 3$.     
Assume that for each prime $p$ dividing $\Delta(f)$ and not dividing $ab$, we have that $p$ does not divide $n+1$. Then a prime factor $p$ of $\Delta(f)$ does not divide the {\rm index}($f$) if and only if $p$ satisfies one of the following conditions:
 \begin{enumerate}[label=(\roman*), font=\normalfont]
 \item \label{JKY:I1} when $p\mid b$, then $p^2\nmid b$;
 \item \label{JKY:I2} when $p\nmid b$ and $p\mid a$ with $n=p^js$, $p\nmid s$, $j\ge 1$, then either 
 \[p\mid a_1 \ \mbox{and} \ p\nmid b_1 \quad \mbox{or} \quad p\nmid a_1((-b_1)^s+ba_1^s),\]
 where $a_1=\frac{ab}{p}$ and $b_1=\frac{b+(-b)^{p^j}}{p}$;
 \item \label{JKY:I3} when $p\nmid ab$, then $p^2\nmid \Delta(f)$.   
 \end{enumerate}
\end{thm}  
The next theorem targets the monogenicity of power-compositional polynomials, and will be used in the proof of Theorem \ref{Thm:Main3}. 
 \begin{thm}{\rm \cite{KKR}}\label{Thm:KKR}
   Let $f(x)\in \Z[x]$ be monic, and let $k\ge 2$ be an integer such that ${\mathfrak F}(x):=f(x^k)$ is irreducible over $\Q$. Then ${\mathfrak F}(x)$ is monogenic if and only if all of the following conditions are true:
   \begin{enumerate}
     \item \label{I1:KKR} $f(x)$ is monogenic,
     \item \label{I2:KKR} $p$ does not divide the {\rm index}(${\mathfrak F}$) for every prime divisor $p$ of $k$, 
     \item \label{I3:KKR} $f(0)$ is squarefree.
   \end{enumerate}  
 \end{thm}
The following two lemmas will also be helpful in establishing our results.   
 \begin{lemma}\label{Lem:Irreduc}
 $\FF_{n,b,m,k}(x)$ is irreducible over $\Q$.
\end{lemma}
\begin{proof}
  Since $b\ge 2$ is squarefree by \eqref{Defs1}, it follows that $\FF_{n,b,m,k}(x)$ is Eisenstein with respect to any prime divisor of $b$, and is therefore irreducible over $\Q$.
\end{proof}

The next lemma follows from \cite[Corollary 2.8]{HJKodai}.
 \begin{lemma}\label{Lem:Discs} 
 $\abs{\Delta(\FF_{n,b,m,k})}=b^{nk-1}k^{nk}\abs{\D}^k$ with $\D\in \Z$. 
  \end{lemma}

\section{The Proof of Theorem \ref{Thm:Main1}} 
\begin{proof}
Note that $\FF_{2,b,m,k}(x)$ is irreducible over $\Q$ by Lemma \ref{Lem:Irreduc}. 
By Lemma \ref{Lem:Discs} and \eqref{Defs1}, we have that 
  \begin{equation}\label{Delta and d}
  \abs{\Delta(\FF_{2,b,m,k})}=b^{2k-1}k^{2k}(bm^4-4)^k \quad \mbox{and} \quad d=\frac{bm^4-4}{\gcd(2,m)^2}.
   \end{equation} Let $p$ be a prime divisor of $\Delta(\FF_{2,b,m,k})$. We invoke Theorem \ref{Thm:JKS2} with 
  \begin{equation}\label{NMAB}
  N:=2k,\quad M:=k, \quad A:=bm^2 \ \mbox{ and } \ B:=b,
   \end{equation} to derive necessary and sufficient conditions for the monogenicity of $\FF_{2,b,m,k}(x)$. We say that a particular condition in Theorem \ref{Thm:JKS2} is {\em vacuous} if the divisibility criteria for $p$ in that condition are impossible under constraints imposed by \eqref{Defs1}, \eqref{Defs2} and \eqref{NMAB}.

   It is easy to see from \eqref{NMAB} that the divisibility criteria for $p$ render condition \ref{JKS:C3} of Theorem \ref{Thm:JKS2} vacuous, while condition \ref{JKS:C4} is vacuous since $\rad(k)$ divides $m$ from \eqref{Defs1}. 
      We also have by \eqref{NMAB} that the divisibility criterion for $p$ in condition \ref{JKS:C5} of Theorem \ref{Thm:JKS2} is that $p\nmid b^2m^2k$.  
   Since $p\mid \Delta(\FF_{2,b,m,k})$, we deduce from \eqref{Delta and d} that $p\mid (bm^4-4)$ and that $p\ge 3$. Thus, $p^2\nmid (bm^4-4)$ from \eqref{Delta and d}, since $d$ is squarefree by \eqref{Defs2}. Consequently, $p$ satisfies condition \ref{JKS:C5} of Theorem \ref{Thm:JKS2}. Similarly, if $p\mid bm^2$ and $p\mid b$, then $p^2\nmid b$ since $b$ is squarefree by \eqref{Defs1}, so that $p$ satisfies condition \ref{JKS:C1} of Theorem \ref{Thm:JKS2}. Therefore, we have shown that only condition \ref{JKS:C2} of Theorem \ref{Thm:JKS2} needs to be checked to determine the monogenicity of $\FF_{2,b,m,k}(x)$. That is, we only have to check prime divisors of $\gcd(\Delta(\FF_{2,b,m,k}),m)$. Since $\gcd(b,m)=1$ and $\rad(k)$ divides $m$ from \eqref{Defs1}, it follows from \eqref{Delta and d} that we need to examine precisely the prime divisors of $k$, in addition to the prime $p=2$, if $2\mid m$ and $2\nmid k$. 
   
   Suppose first that $p\mid k$. Then $p\mid m$ since $\rad(k)$ divides $m$ from \eqref{Defs1}, and we see from condition \ref{JKS:C2} of Theorem \ref{Thm:JKS2} that $a_2=bm^2/p \equiv 0\pmod{p}$. Thus, $p\nmid {\rm index}(\FF_{2,b,m,k})$ if and only if 
   \begin{equation}\label{b1}
   b_1=\frac{b+(-b)^{p^j}}{p}\not \equiv 0 \pmod{p},
    \end{equation} where $p^j\mid \mid 2k$ with $j\ge 1$. If $p=2$, then $2\nmid b$ since $\gcd(b,m)=1$ from \eqref{Defs1}. It follows that $2\nmid {\rm index}(\FF_{2,b,m,k})$ if and only if $b\equiv 1 \pmod{4}$ since 
   \[2b_1=b+(-b)^{2^j}\equiv \left\{\begin{array}{cl}
     2 \pmod{4} & \mbox{if $b\equiv 1 \pmod{4}$}\\
     0 \pmod{4} & \mbox{if $b\equiv 3 \pmod{4}$,}
   \end{array}\right.\] from \eqref{b1}. Suppose then that $p\ge 3$. Then $p\nmid b$ since $\gcd(b,m)=1$ from \eqref{Defs1}. We deduce from \eqref{b1} that $p\nmid {\rm index}(\FF_{2,b,m,k})$ if and only if 
   \[b^{p^j}-b\not \equiv 0 \pmod{p^2},\] which is equivalent to $b^p-b\not \equiv 0 \pmod{p^2}$ since $\phi(p^2)=p(p-1)$.    
   
   The case when $p=2$ such that $2\mid m$ and $2\nmid k$ is identical to the situation when $2\mid k$, and we omit the details. 
    \end{proof}

\section{The Proof of Theorem \ref{Thm:Main2}}
\begin{proof} 
Note that $\FF_{n,b,m,1}(x)=f_{n,b,m}(x)$ is irreducible over $\Q$ by Lemma \ref{Lem:Irreduc}. 
We utilize Theorem \ref{Thm:JKY} with $a:=m^2$ to derive necessary and sufficient conditions for the monogenicity of $\FF_{n,b,m,1}(x)$. 
From Lemma \ref{Lem:Discs}, 
\begin{equation}\label{k=1}
\abs{\Delta(\FF_{n,b,m,1})}=b^{n-1}\abs{\D}.  
\end{equation} Let $p$ be a prime divisor of $\Delta(\FF_{n,b,m,1})$. 

If $p\mid b$, then $p^2\nmid b$ since $b$ is squarefree by \eqref{Defs1}. Thus, $p\nmid {\rm index}(\FF_{n,b,m,1})$ by Theorem \ref{Thm:JKY}. So, assume that $p\nmid b$. Then $p\mid \D$ from \eqref{k=1}, and it is easy to see from the definition of $\D$ in \eqref{Defs2} that $p\mid m$ if and only if $p\mid n$.  
Suppose that $p\mid n$ and let $p^j\mid \mid n$ with $j\ge 1$. Observe then that  $a_1=bm^2/p\equiv 0 \pmod{p}$ in condition \ref{JKY:I2} of Theorem \ref{Thm:JKY}. We deduce that $p\nmid {\rm index}(\FF_{n,b,m,1})$ if and only if 
\begin{equation}
  b+(-b)^{p^j} \not \equiv 0 \pmod{p^2},
\end{equation} which is equivalent to $b^{p-1}-1 \not \equiv 0 \pmod{p^2}$ if $p\ge 3$, and $b\equiv 1 \pmod{4}$ if $p=2$. 
Suppose next that $p\nmid bm$. To apply Theorem \ref{Thm:JKY}, we must show that $p\nmid (n+1)$. Assume, by way of contradiction, that $p\mid (n+1)$, and let $n+1=py$, for some $y\in \Z$. Since $p\mid \D$, we see that $p\mid (1-bm^{2n})$ from \eqref{Defs2}, so that $p\mid (1+bm^{2n}n)$ since $n\equiv -1\pmod{p}$. Let $1-bm^{2n}=pz$ and $1+bm^{2n}n=pw$, for some $z,w\in \Z$. Then
\[\D=\frac{n^n(pz)^{n+1}+bm^{2n}(py)^{n+1}}{(pw)^2}=p^{n-1}\left(\frac{n^nz^{n+1}+bm^{2n}y^{n+1}}{w^2}\right),\] where $\frac{n^nz^{n+1}+bm^{2n}y^{n+1}}{w^2}\in \Z$. Since $p\nmid bm$ and $n\ge 3$, we deduce that $p^2\mid d$, contradicting the fact that $d$ is squarefree from \eqref{Defs2}. Thus, $p\nmid (n+1)$, and we conclude from Theorem \ref{Thm:JKY} that $p\nmid {\rm index}(\FF_{n,b,m,1})$ since $d$ is squarefree, which completes the proof of the theorem. 
\end{proof}

\section{The Proof of Theorem \ref{Thm:Main3}}
\begin{proof}
In light of Lemma \ref{Lem:Irreduc}, with 
\[f(x):=f_{n,b,m}(x) \quad \mbox{and} \quad {\mathfrak F}(x):=\FF_{n,b,m,k}(x)=f_{n,b,m}(x^k),\] we use Theorem \ref{Thm:KKR} to derive necessary and sufficient conditions for the monogenicity of $\FF_{n,b,m,k}(x)$ when $k\ge 2$. Item \eqref{I3:KKR} of Theorem \ref{Thm:KKR} is satisfied since $f_{n,b,m}(0)=b$ is squarefree by \eqref{Defs1}. Necessary and sufficient conditions for the monogenicity of $f_{n,b,m}(x)$ are given in Theorem \ref{Thm:Main2}, which addresses item \eqref{I1:KKR} of Theorem \ref{Thm:KKR}. Thus, to complete the necessary and sufficient conditions for the monogenicity of $\FF_{n,b,m,k}(x)$ when $k\ge 2$, we see, by item \eqref{I2:KKR} of Theorem \ref{Thm:KKR}, that we only need to determine conditions under which no prime divisor of $k$ divides index($\FF_{n,b,m,k}$). To accomplish this task, we use Theorem \ref{Thm:Dedekind} with $T(x):=\FF_{n,b,m,k}(x)$ and $p$ a prime divisor of $k$. Observe that $p\mid m$ and $p\nmid b$ from \eqref{Defs1}. 
Since $g_p(x)$ is irreducible in $\F_p[x]$, it follows that $\overline{T}(x)=g_p(x)^{p^{\nu_p(k)}}$, so that we can let $g(x):=g_p(x)$ and $h(x):=g_p(x)^{p^{\nu_p(k)}-1}$ in Theorem \ref{Thm:Dedekind}. Hence,
\begin{align*}
  F(x)&=\frac{g(x)h(x)-T(x)}{p}=\frac{g_p(x)^{p^{\nu_p(k)}}-\FF_{n,b,m,k}(x)}{p}\\
  &=\frac{\left(x^{nk/p^{\nu_p(k)}}+b\right)^{p^{\nu_p(k)}}-\left(x^{kn}+b\sum_{j=1}^{n}(m^2x^k)^{n-j}\right)}{p}\\
  &=\frac{b^{p^{\nu_p(k)}}-b}{p}+
  \sum_{j=1}^{p^{\nu_p(k)}-1}\frac{\binom{p^{\nu_p(k)}}{j}}{p}\left(x^{nk/p^{\nu_p(k)}}\right)^{p^{\nu_p(k)}-j}b^j
  -b\sum_{j=1}^{n-1}\left(\frac{m^2}{p}x^k\right)^{n-j}.
    \end{align*} Thus, since $p\mid m$, we have that 
    \begin{equation*}\label{Fbar}
    F(x)\equiv \frac{b^{p^{\nu_p(k)}}-b}{p}+\sum_{j=1}^{p^{\nu_p(k)}-1}\frac{\binom{p^{\nu_p(k)}}{j}}{p}
  \left(x^{nk/p^{\nu_p(k)}}\right)^{p^{\nu_p(k)}-j}b^j \pmod{p}.
  \end{equation*} Suppose that $g_p(\alpha)\equiv 0 \pmod{p}$. Then $\alpha^{nk/p^{\nu_p(k)}}\equiv -b\pmod{p}$, and therefore,
  \begin{align*}\label{alt}
  F(\alpha)&\equiv \frac{b^{p^{\nu_p(k)}}-b}{p}+\sum_{j=1}^{p^{\nu_p(k)}-1}\frac{\binom{p^{\nu_p(k)}}{j}}{p}
  \left(\alpha^{nk/p^{\nu_p(k)}}\right)^{p^{\nu_p(k)}-j}b^j \pmod{p}\\ 
  &\equiv \frac{b^{p^{\nu_p(k)}}-b}{p}+b^{p^{\nu_p(k)}}\sum_{j=1}^{p^{\nu_p(k)}-1}\frac{\binom{p^{\nu_p(k)}}{j}}{p}
  (-1)^{p^{\nu_p(k)}-j}\pmod{p}\\ 
  &\equiv \left\{\begin{array}{cl}
   \frac{b^{p^{\nu_p(k)}}-b}{p} \pmod{p} & \mbox{if $p\ge 3$}\\
    \frac{b^{2^{\nu_2(k)}}-b}{2}+1 \pmod{2} & \mbox{if $p=2$},
  \end{array}\right. \\ 
  &\equiv \left\{\begin{array}{cl}
   \frac{b^{p}-b}{p} \pmod{p} & \mbox{if $p\ge 3$}\\
     1 \pmod{2} & \mbox{if $p=2$ and $b\equiv 1 \Mod{4}$}\\
     0 \pmod{2} & \mbox{if $p=2$ and $b\equiv 3 \Mod{4}$}.
    \end{array}\right. 
   \end{align*} 
   Combining these computations with Theorem \ref{Thm:Main2} yields the theorem.   
    \end{proof}

\end{document}